\documentclass{amsart}    

\usepackage{graphicx}

\usepackage{amssymb}
\usepackage{amsmath} 
\usepackage{amscd} 
\usepackage[all]{xypic}
\usepackage{url}
\usepackage{color}
\usepackage{algpseudocode}

\usepackage{algorithm}

\newcommand{\N}{\mathcal{N}}

\newcommand{\fm}{\mathfrak m}
\newcommand{\fn}{\mathfrak n}

\newcommand{\R}{\mathbb{R}}

\newcommand{\M}{\mathcal{M}}

\newtheorem{prop*}{Proposition}
\newtheorem{thm}[subsection]{Theorem}
\newtheorem{thm*}{Theorem}
\newtheorem{lemma*}{Lemma}

\newtheorem{cor*}{Corollary}
\newtheorem{rem}[subsection]{Remark}
\newtheorem{rem*}{Remark}
\newtheorem{def*}{Definition}

\makeatletter
\@namedef{subjclassname@2020}{%
  \textup{2020} Mathematics Subject Classification}
\makeatother

\begin{document}

\title[A note on GA and control problems with SO(3)--symmetries]{A note on geometric algebras and control problems with SO(3)--symmetries}

\keywords{
local control and optimality, Carnot groups, symmetries, sub--Riemannian geodesics, geometric algebras}

\subjclass[2020]{15A67, 53C17}

\author{
Jaroslav Hrdina, Ale\v s N\' avrat, Petr Va\v s\'ik and Lenka Zalabov\'a
}

\address{
JH, AN, PV: Institute of Mathematics, 
	Faculty of Mechanical Engineering,  Brno University of Technology,
	Technick\' a 2896/2, 616 69 Brno, Czech Republic; LZ: Institute of Mathematics, 
	Faculty of Science, University of South Bohemia, 
	Brani\v sovsk\' a 1760, 370 05 \v Cesk\' e Bud\v ejovice, and
    Department of Mathematics and Statistics, 
	Faculty of Science, Masaryk University,
	Kotl\' a\v rsk\' a 2, 611 37 Brno, Czech Republic
}

\email{hrdina@fme.vutbr.cz, navrat.a@fme.vutbr.cz
}	
\email{vasik@fme.vutbr.cz, lzalabova@gmail.com
}	
\thanks{
The first three authors were supported by the grant no. FSI-S-20-6187. Fourth author is supported by the grant no. 20-11473S Symmetry and invariance in analysis, geometric modeling and control theory from the Czech Science Foundation. Finally, we thank the referee for valuable comments.
}

\maketitle

\thispagestyle{empty}

\vspace{7pt}

\begin{abstract}
We study the role of symmetries in control systems through the geometric algebra approach. We discuss two specific control problems on  Carnot groups of step $2$ invariant with respect to the action of $SO(3)$. We understand the geodesics as the curves in suitable geometric algebras which allows us to assess a new algorithm for the local control.
\end{abstract}

\section{Introduction}

Geometric control theory uses geometric methods to control various mechanical systems \cite{ju,bloch}. We use the methods of sub--Riemannian geometry and Hamiltonian concept \cite{AS,ABB}. As a reasonable starting point, we consider mechanisms moving in the plane, typically wheeled mechanisms like cars (with or without trailers) or robotic snakes \cite{hvnm, I04}. The movement of a planar mechanisms is always invariant with respect to the action of the Euclidean group $SE(2)$. As the prototypes of planar mechanisms, we choose those consisting of the body in the shape of a triangle and three legs connected to the vertices of the  body by joints of various types and combinations. Although such mechanisms have almost the same shape, the configuration spaces may differ. In particular, possible motions of the mechanism induce a specific filtration in the configuration space. We present two examples that carry the filtration $(3,6)$ and $(4,7),$ respectively \cite{I04, hz}. 

    \begin{figure} 
	\begin{center}
	\includegraphics[height=45mm]{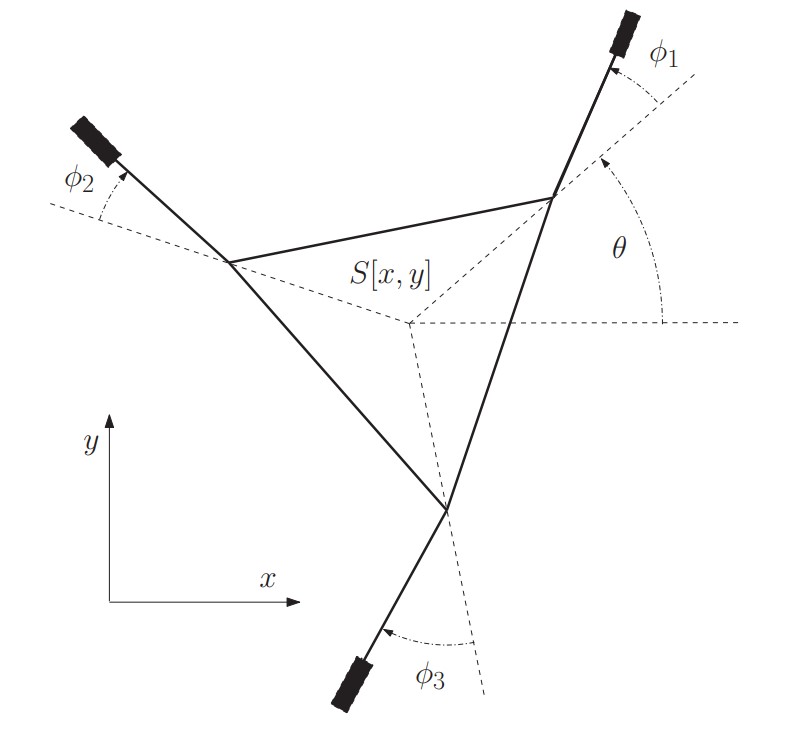}
	\hspace{1cm}
	\includegraphics[height=45mm]{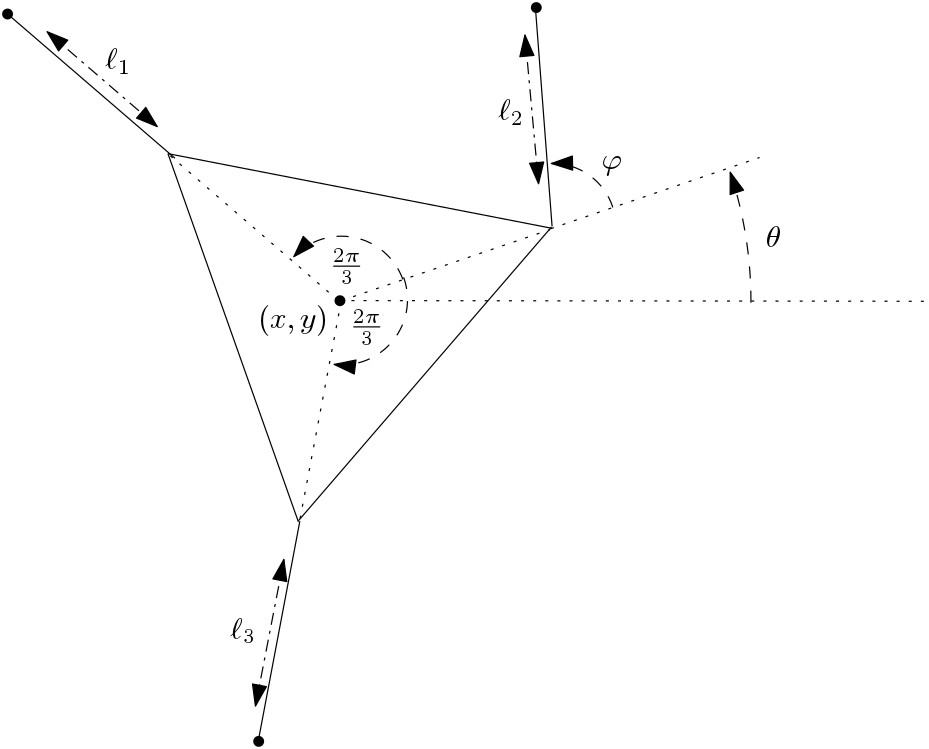} 
	\caption{Generalized trident snakes}
	\end{center}
	\end{figure}
	
To control the mechanisms locally, we consider the nilpotent approximations of the  original control systems \cite{b96}. Although the configuration spaces and their approximations have the same filtration, the approximations form Carnot groups that are generally endowed with more symmetries  \cite{mya1}. One gets the symmetries generated by the right--invariant vector fields and there may be additional symmetries acting non--trivially on the distribution. Our Carnot groups of filtrations $(3,6)$ and $(4,7)$ carry subgroups of the symmetries isomorphic to $SO(3)$  \cite{J,hz}. This observation leads to the idea of the local control in geometric algebra approach. 

We reformulate the control problems in the concept of geometric algebras  $\mathbb G_3$ and $\mathbb G_4$ \cite{hil1,per,Lounesto}. We use the natural $SO(3)$--invariant operations in geometric algebras to reduce the set of geodesics to a simpler set of curves in the geometric algebra  \cite{hzn2}. Namely, each geodesic is a linear combination of orthogonal vectors, and  $SO(3)$ acts on the geodesics by means of the action on the appropriate orthonormal system of vectors. So it is sufficient to study geodesics for one fixed orthonormal basis, i.e. we can study just geodesics in the moduli space over the action of $SO(3)$. 

We present the local control algorithm for finding geodesics passing through the origin and an arbitrary point in its neighbourhood. The algorithm is based on the use of rotors in order to relate two orthogonal bases. We provide an efficient method to such comparison using geometric algebras. We illustrate our algorithm on two specific examples.

\section{Nilpotent control problems} \label{sec2}
We focus on two control problems such that their symmetry groups contain $SO(3)$ as their subgroups.  The first system has the growth vector $(3,6)$ and the other one has the growth vector $(4,7)$ \cite{mya1,hz}. 

\subsection{Control problems on Carnot groups of step $2$}
By nilpotent control problems we mean the invariant control problems on Carnot groups and we consider the Carnot groups $G$ of step $2$ with the filtration $(m,n)$ \cite{ABB,Rizzi, talianky}.  If we denote the local coordinates by $(x,z) \in \mathbb R^m \oplus \mathbb R^{n-m}$, we can model the corresponding Lie algebra $\mathfrak g$ of vector fields  
\begin{align}
\begin{split} \label{vp}
 X_i&= {\partial_{x_i}} -\frac{1}{2} \sum_{l=1}^{n-m}\sum_{j=1}^m c_{ij}^l  x_j {\partial_{z_l}}\ \ \  j=1, \dots ,m  \\ \
X_{m+j}&=  {\partial_{z_{j}}} \ \ \hspace{2.9cm}  j=1, \dots ,m-n  ,
\end{split}
\end{align}
where $c_{jl}^k$ are the structure constants of the Lie algebra $\mathfrak g$ and 
 the symbol $\partial$ stands for partial derivative. We discuss 
the related optimal control problem   
\begin{align} \dot q (t)
= u_1 X_1 + \cdots +u_m X_m
\label{control}
\end{align}
for $t>0$ and $q$ in $G$ and the control $u=(u_1(t),\dots,u_m(t)) \in \R^m$ with the boundary condition $q(0)=q_1,$ $q(T)=q_2$ for fixed points $q_1, q_2 \in G$,
where we minimize the cost functional ${1 \over 2}\int_0^T (u_1^2+\cdots+u_m^2) \mathrm dt$.
The solutions $q(t)$ then correspond to the sub--Riemannian geodesics, i.e. admissible curves parametrized by a constant speed whose sufficiently small arcs are the length minimizers.

We use the Hamiltonian approach to this control problem \cite{ABB}. 
There are no strict abnormal extremals for the step $2$ Carnot groups, so we focus on the normal geodesics and address them just as geodesics. 
The left--invariant vector fields $X_i$, $i=1,\dots,m$ form a basis of $TG$ and determine the left--invariant coordinates on $G$.
We define the corresponding left--invariant coordinates $h_i$, $i=1,\dots,m$ and $w_i$, $i=1,\dots, n-m$ on the fibres of $T^*G$ by $h_i (\lambda )=  \lambda(X_i) $ and  $w_i (\lambda )=  \lambda(X_{m+i})$, for arbitrary $1$--forms $\lambda$ on $G$. Thus we use $(x_i,w_i)$ as the global coordinates on $T^*G$.

The geodesics are exactly the projections of normal Pontryagin extremals, i.e. the integral curves of the left--invariant normal Hamiltonian 
\begin{align} \label{ham}
H={1 \over 2}(h_1^2+h_2^2+ \cdots + h_m^2),
\end{align}
on $G$. Assume that $\lambda(t)=(x_i(t),z_i(t),h_i(t),w_i(t))$ in $T^*G$ is a normal extremal. 
Then the controls $u_j$ to the system  \eqref{control} satisfy $u_j(t)=h_j(\lambda(t))$ and the base system takes the form of
\begin{align}
\begin{split} \label{horiz} 
 \dot x_i&= h_i,\ \ \  i=1, \dots ,m  \\ 
\dot z_j&= -\frac{1}{2} \sum_{i=1}^{m} c_{ik}^j h_i x_k ,\ \ \  j=1, \dots ,n-m 
\end{split}
\end{align}
for $q=(x_i,z_i)$. Using $u_j(t)=h_j(\lambda(t))$ and the equation $\dot \lambda(t)=\vec{H}(\lambda(t))$ for the normal extremals, we write the fiber system as
\begin{align}
\begin{split}
\dot h_i&= -\sum_{l=1}^{m-n} \sum_{j=1}^m c_{ij}^l h_j w_l,\ \ \  i=1, \dots ,m,  \\ \label{ver_obec2}
\dot w_{j}&=0, \ \ \hspace{2.5cm}  j=1, \dots ,n-m  ,
\end{split}
\end{align}
where $c_{ij}^l$ are the structure constants of the Lie algebra $\mathfrak g$ for the basis $X_i$.
The solutions $w_i$, $i=1,\dots, n-m $ are constants that we denote by
\begin{align} \label{h567}
w_1=K_1,\ \dots ,w_{n-m}=K_{n-m}.
\end{align} 
If $K_1=\cdots=K_{n-m}=0$ then $h(t)=h(0)$ is a constant and the geodesic
 $(x_i(t),z_i(t))$ is a line in $G$ such that $z_i(t)=0$. If at least one of $K_i$ is non--zero, the first part of the fibre system \eqref{ver_obec2} forms a homogeneous system of ODEs 
 $
 \dot h = - \Omega h
 $
 with constant coefficients for $h=(h_1,\dots, h_m)^T$ and the system matrix $\Omega$. 
Its solution is given by  
$
h(t)= e^{-t \Omega} h(0)
$, where  $h(0)$ is the initial value of the vector $h$ at the origin.

\subsection{Left--invariant control problem with the growth vector $(3,6)$} 
\label{sec21} 
Let us consider three vector fields on $\R^6$ with the local coordinates $(x_1,x_2,x_3,z_1,z_2,z_3)$ in the form
\begin{align}
\begin{split}
X_1&=\partial_{x_1}+{x_3 \over 2}\partial_{z_2}-{x_2 \over 2}\partial_{z_3},\\
X_2&=\partial_{x_2}+{x_1 \over 2}\partial_{z_3}-{x_3 \over 2}\partial_{z_1},\\
X_3&=\partial_{x_3}+{x_2 \over 2}\partial_{z_1}-{x_1 \over 2}\partial_{z_2}.
\label{vf1}
\end{split}
\end{align} 
The only non--trivial Lie brackets are 
\begin{align}
\begin{split}
X_4=[X_1,X_2]= \partial_{z_3}, \ \  X_5=[X_1,X_3]= -\partial_{z_2}, \ \ X_6=[X_2,X_3]= \partial_{z_1}.
\end{split}
\end{align}
 These six vector fields determine a step $2$ nilpotent Lie algebra $\fm$ with the multiplication table given by Table \ref{tab1}.
 
\begin{table}[h] \label{liealg1}
\begin{center}
\begin{tabular}{ |c|| c | c | c |c |c |c |} 
\hline
$\fm$  & $X_1$ & $X_2$ & $X_3$ & $X_4$ & $X_5$ & $X_6$  \\
\hline \hline
$X_1$ & $0$ & $X_4$ & $X_5$ &$0$ &$0$ &$0$    \\ 
\hline
$X_2$ & $-X_4$ & $0$ & $X_6$ &$0$ &$0$ &$0$    \\ 
\hline
$X_3$ & $-X_5$ & $-X_6$ & $0$ &$0$ &$0$ &$0$    \\ 
\hline
$X_4$ & $0$ & $0$ & $0$ &$0$ &$0$ &$0$    \\ 
\hline
$X_5$ & $0$ & $0$ & $0$ &$0$ &$0$ &$0$    \\ 
\hline
$X_6$ & $0$ & $0$ & $0$ &$0$ &$0$ &$0$    \\ 
\hline
\end{tabular}
\end{center}
\caption{Lie algebra $\fm$}
\label{tab1}
\end{table}
There is a Carnot group $M$ such that the fields $X_i$, $i=1\dots,6$ are left--invariant for the corresponding group structure. When identified with $\mathbb R^6= \mathbb R^3 \oplus  \R^3$, the group structure on $M$ reads as
\begin{align}\label{grupa}
(x,z)\cdot (x',z')=(x+x',z+z'+{1 \over 2}x \times x')
 \end{align}
for $x=(x_1,x_2,x_3)$ and $z=(z_1,z_2,z_3)$, where $\times$ stands for the vector product on $\R^3$.
In particular, $\M=\langle X_1,X_2,X_3\rangle$ forms a $3$--dimensional left--invariant distribution on $M$. 
We define the left--invariant sub--Riemannian metric $g_M$ on $\M$ by declaring  $X_1,X_2,X_3$ orthonormal.

The geodesics of the control problem  are the solutions to the control system \eqref{horiz},\eqref{ver_obec2}, with $(m,n)=(3,6)$ and one can read the structure constants in Table \ref{tab1}. Hence, the fibre system is given by
$w_1=K_1, w_2=K_2, w_3=K_3$ for the constants $K_1,K_2,K_3$ and
$
\dot h = - \Omega h$ for 
 $h=(h_1,h_2,h_3)^T$ and 
\begin{align}
\Omega=
\left( \begin{smallmatrix} 
0 & K_1 & K_2 \\
-K_1 & 0 & K_3 \\
-K_2 &-K_3 & 0
\end{smallmatrix} \right)
\label{matrix}.
\end{align}
Its solution is given by the exponential  
$
h(t)= e^{-t \Omega} h(0)
$, where  $h(0)$ is the initial value of the vector $h$ at the origin.  We write an explicit formula for the general solution in terms of eigenvectors of \eqref{matrix}.  If at least one of the  constants $K_i$ is non--zero, the kernel of $\Omega$, i.e. zero--eigenspace, is one--dimensional, generated by  the vector $(K_3,K_2,K_1)^T$. Its orthogonal complement corresponds to the sum of eigenspaces to the eigenvalues $\pm iK$, where  $K:=\sqrt{K_1^2+K_2^2+K_3^2}$, and is generated by the vectors $(-K_1K_3,-K_1K_2,K_2^2+K_3^2)\pm i (K_2,-K_3,0)$. Thus solution to the fibre system can be written as
\begin{align}  \label{h36}
h(t)=(C_1\cos(Kt)-C_2\sin(Kt))v_1+(C_1\sin(Kt)+C_2\cos(Kt))v_2 +  C_3 v_3,
\end{align} 
where   $v_1,v_2,v_3$ is the eigenspace--adapted real orthonormal basis 
$$
v_1= \frac{1}{K \sqrt{K_2^2+K_3^2}}  \begin{pmatrix} -K_1K_3 \\ K_1 K_2 \\ K_2^2+K_3^2\end{pmatrix},
v_2=\frac{1}{ \sqrt{K_2^2+K_3^2}} \begin{pmatrix}  -K_2 \\  - K_3 \\ 0\end{pmatrix}, 
v_3=\frac{1}{K} \begin{pmatrix} K_3 \\- K_2 \\ K_1\end{pmatrix}$$
and $C_1,C_2,C_3$ are the constants that satisfy the level set condition $H=1/2$, i.e. $\|h(t)\|=1$, that reads
$C_1^2+C_2^2+C_3^2=1.$ 
Let us note that the choice $C_1=C_2=0$ leads to the constant solutions that are irrelevant as the control functions. Thus we assume that at least one of the constants $C_1,C_2$ is non--zero.
 
Let us emphasize that the base system \eqref{horiz} can be written in terms of a vector product as follows
\begin{align} \begin{split} \label{rr1} 
&\dot x = h, \\
&\dot z = {1 \over 2}x \times h
\end{split}
\end{align}
 for vectors $x=(x_1,x_2,x_3)^T$ and $z=(z_1,z_2,z_3)^T$. One obtains the general solution by substituting \eqref{h36} for $h$ and by consequent direct integration.
We are interested in the solutions passing through the origin, i.e. we impose the initial condition 
\begin{align}\label{ic1}
x_i(0)&=0,z_i(0)=0,\quad i=1,2,3.
\end{align}
However, it may be difficult to find the integration constants giving the geodesics through a fixed target point.


\subsection{Left--invariant control problem with the growth vector $(4,7)$} \label{sec22}
Let us consider four vector fields on $\R^7$ with the local coordinates $(x,\ell_1,\ell_2,\ell_3,y_1,y_2,y_3)$ in the form
\begin{align}
\begin{split}
&Y_0=\partial_x-{\ell_1 \over 2}{\partial_{y_1}}-{\ell_2 \over 2}{\partial_{y_2}}-{\ell_3 \over 2}{\partial_{y_3}},\\
&Y_1=\partial_{\ell_1}+{x \over 2} \partial_{y_1}, \ \ 
Y_2=\partial_{\ell_2}+{x \over 2} \partial_{y_2}, \ \ 
Y_3=\partial_{\ell_3}+{x \over 2} \partial_{y_3}.
 \label{vf}
\end{split}
\end{align}
The only non--trivial Lie brackets are 
\begin{align}
\begin{split}
Y_{4}=[Y_0,Y_1]= \partial_{y_1}, \ \ Y_{5}=[Y_0,Y_2]= \partial_{y_2}, \ \ Y_{6}=[Y_0,Y_3]= \partial_{y_3}
\end{split}.
\end{align}
 These seven fields determine a step $2$ nilpotent Lie algebra $\fn$ with the multiplication table given by Table \ref{tab}.
\begin{table}[h] \label{liealg}
\begin{center}
\begin{tabular}{ |c|| c | c | c |c |c |c |c |} 
\hline
 $\fn$ & $Y_0$ & $Y_1$ & $Y_2$ & $Y_3$ & $Y_{4}$ & $Y_{5}$ & $Y_{6}$ \\
\hline \hline
$Y_0$ & $0$ & $Y_{4}$ & $Y_{5}$ &$Y_{6}$ &$0$ &$0$ &$0$   \\ 
\hline
$Y_1$ & $-Y_{4}$ & $0$ & $0$ &$0$ &$0$ &$0$ &$0$   \\ 
\hline
$Y_2$ & $-Y_{5}$ & $0$ & $0$ &$0$ &$0$ &$0$ &$0$   \\ 
\hline
$Y_3$ & $-Y_{6}$ & $0$ & $0$ &$0$ &$0$ &$0$ &$0$   \\ 
\hline
$Y_{4}$ & $0$ & $0$ & $0$ &$0$ &$0$ &$0$ &$0$   \\ 
\hline
$Y_{5}$ & $0$ & $0$ & $0$ &$0$ &$0$ &$0$ &$0$   \\ 
\hline
$Y_{6}$ & $0$ & $0$ & $0$ &$0$ &$0$ &$0$ &$0$   \\
\hline 
\end{tabular}
\end{center}
\caption{Lie algebra $\fn$
}
\label{tab}
\end{table}

There is a Carnot group $N$ such that the fields $Y_i$, $i=1,\dots,7$ are left--invariant for the corresponding group structure. 
The group structure on $N$, when identified with $\mathbb R^7= \mathbb R \oplus \mathbb R^3 \oplus \mathbb R^3$, yields
\begin{align}\label{grupa2}
(x,\ell,y)\cdot (x',\ell',y')=(x+x',\ell+\ell',y+y'+{1 \over 2}\ell \times \ell').
\end{align}
for $\ell=(\ell_1,\ell_2,\ell_3)$ and  
 $y=(y_1,y_2,y_3)$.
In particular, $\N=\langle Y_0,Y_1,Y_2,Y_3\rangle$ forms a $4$--dimensional left--invariant distribution on $N$. 
Moreover, there is a natural decomposition
\begin{align}
\label{decom}
\N=\langle Y_0 \rangle \oplus \langle Y_1, Y_2, Y_3 \rangle 
\end{align}
into a $1$--dimensional distribution and a $3$--dimensional involutive distribution,  both left--invariant.
We define the left--invariant sub--Riemannian metric $g_N$ on $\N$ by declaring $Y_0$, $Y_1$, $Y_2$, $Y_3$ orthonormal.

The geodesics of the control problem  are solutions to the control system \eqref{horiz},\eqref{ver_obec2}, with $(m,n)=(4,7)$ and  we read the structure constants in Table \ref{tab}. Hence, the first part of the fibre system \eqref{ver_obec2} is given by
$w_1=K_1, w_2=K_2, w_3=K_3$, where $K_1,K_2,K_3$ are constants. Second part of the fibre system takes form
$
\dot h = - \Omega h,
$
where $h:=(h_0,h_1,h_2,h_3)^T$ and
\begin{align}
\Omega=
\left( \begin{smallmatrix} 
0 & K_1 & K_2 &K_3 \\
-K_1 & 0 & 0&0 \\
-K_2 & 0 & 0 & 0\\
-K_3 & 0 & 0 & 0
\end{smallmatrix} \right).
\label{matrix2}
\end{align}
Its solution is given by  
$
h(t)= e^{-t \Omega} h(0)
$, where  $h(0)$ is the initial value of the vector $h$ at the origin and we write its explicit form in terms of the eigenvectors of \eqref{matrix2}. If $K_1=K_2=K_3=0$ then $h(t)=h(0)$ is a constant and the geodesic   $(x(t),\ell_i(t),y_i(t))$ is a line in $N$ such that $y_i=0$. 
If at least one of the constants $K_i$ is non--zero, 
the kernel of $\Omega$, i.e. zero--eigenspace, is two--dimensional and is generated by the vectors $(0,-K_3,0,K_1)^T$ and $(0,-K_2,K_1,0)^T$. Its orthogonal complement corresponds to the sum of the eigenspaces  to the eigenvalues $\pm iK$, where $K:=\sqrt{K_1^2+K_2^2+K_3^2}$, and is generated by the eigenvectors 
$(0,K_1,K_2,K_3)^T  \pm i(K,0, 0,0)^T$. Thus the solution to the vertical system for non--zero $K$ takes form   
\begin{align} 
\begin{split}\label{h1}
h_0&=K(C_2\cos(Kt)-C_1\sin(Kt)) \\ 
\bar h& =K(C_2\sin(Kt)+C_1\cos(Kt))
r_1
+Cr_2
\end{split} 
\end{align}
where  $\bar h= (h_1,h_2,h_3)^T$ and  $r_1,r_2$ are eigenspace--adapted real orthonormal vectors
$$
r_1=\frac{1}{K}\left( \begin{smallmatrix} 
K_1\\
K_2\\
K_3
\end{smallmatrix} \right), \ \ \  
r_2= \frac{1}{C}  \left(C_3   \left( \begin{smallmatrix} 
-K_3\\0
\\
K_1
\end{smallmatrix} \right) +
  C_4 \left( \begin{smallmatrix} 
-K_2\\
K_1\\
0
\end{smallmatrix} \right) \right)
$$
with the constants $C_1,C_2,C_3,C_4$ and the normalization factor 
$$C= \sqrt{(C_3K_3 + C_4K_2)^2+K_1^2(C_3^2 + C_4^2)}.$$ 
The level set condition $\|h(t)\|=1$ reads $C_1^2+C_2^2+C_3^2=1.$
Let us note that the choice $C_1=C_2=0$ leads to the constant solutions that are irrelevant as the control functions. Thus we assume that at least one of the constants $C_1,C_2$ is non--zero.

The base system \eqref{horiz} takes the explicit form of
\begin{align}
\begin{split} \label{xl}
&\dot x=h_0, \\
&\dot{ \ell}=\bar h, \\
&\dot{y}= {1 \over 2}(x  \bar h -h_0 \ell).
\end{split}
\end{align}
We are interested in the solutions passing through the origin, i.e. we impose the initial condition 
\begin{align}\label{ic2}
x(0)=0,\ell_i(0)=0,y_i(0)=0,\quad i=1,2,3.
\end{align} 
By substitution of \eqref{h1}, the system \eqref{xl} can be directly integrated. 
Again, it may be difficult to find the geodesics through a fixed target point.
In the section \ref{GAapproach} we show how the symmetries of the system and the geometric algebra approach is used for finding a geodesic towards a given point.

\subsection{Symmetries of the control systems} \label{sec3}

Symmetries of the control system in question coincide with the symmetries of the corresponding left--invariant sub--Riemannian structure $(M,\M, g_M)$ and $(N,\N, g_N)$, respectively. These are precisely the automorphisms on groups preserving the distributions and sub--Riemannian metrics. 
The group $SO(3)$ acts on $\mathbb R^3$ and preserves the vector product which implies the following statement. 

\begin{prop*} \label{akce}
For each $R \in SO(3)$, the map 
\begin{align}
\begin{split} 
(x,z) \mapsto (Rx, Rz)
\end{split}
\label{action}
\end{align}
maps the geodesics of the system from Section \ref{sec21} starting at the origin to the geodesics starting at the origin.
For each $R \in SO(3)$ the map 
\begin{align}
\begin{split} 
(x,\ell,y) \mapsto (x, R \ell, Ry)
\end{split}
\label{action2}
\end{align}
maps the geodesics of the system from Section \ref{sec22} starting at the origin to the geodesics starting at the origin.
\end{prop*}
\begin{proof} Follows from the 
invariance of \eqref{rr1} and \eqref{xl} with respect to the action of $R \in SO(3)$. 
\end{proof}

\section{Geometric algebra}
The construction of the universal real geometric algebra is well-known \cite{hil1, per, Lounesto, hv}. We provide only a brief description in a special case $\mathbb G_m$ that we use later.  In general, geometric algebras are based on symmetric bilinear forms of arbitrary signature. Here, we deal with the real vector space $\R^m$ endowed with a positive definite symmetric bilinear form $B$ only. 

\subsection{Geometric product}
Let us consider a positive definite symmetric bilinear form $B$ on $\mathbb R^m$  and the associated orthonormal basis $(e_1,\dots, e_m)$ ,  i.e.  
\begin{align*}
B(e_i,e_j)=\begin{cases}
1 &\text{ if } i=j\\
0  &\text{ if } i\neq j
\end{cases}
\; \text{ where } 1\leq i,j\leq m.
\end{align*}
The Grassmann algebra $\Lambda(\mathbb{R}^m)$ is an associative algebra with the anti-symmetric outer product $\wedge$ defined by the rule 
\begin{align*}
e_i \wedge e_j + e_j \wedge e_i =0 \; \text{ for } 1\leq i,j \leq m.
\end{align*}
The Grassmann blade of grade $r$ is  $e_A=e_{i_1} \wedge \cdots \wedge e_{i_r}$, where the multi-index $A$ is a set of indices ordered in the natural way $1\leq i_1\leq \cdots \leq i_r\leq m,$ and we put $e_\emptyset=1.$ Blades of grades $0\leq r \leq m$ form the basis of the 
Grassmann algebra $\Lambda(\mathbb{R}^m)$ and  for the  outer product we have
\begin{align*}
e_j \wedge e_A = \begin{cases}
e_j \wedge e_{i_1} \wedge \cdots \wedge e_{i_r} & \text{ if } j\notin A, \\
0  & \text{ if } j\in A
\end{cases}
\end{align*}
and $1 \wedge e_A=e_A$.
For the vectors from $\mathbb R^m$, the inner product $e_i \cdot e_j= B(e_i,e_j)$ and the outer product $e_i \wedge e_j $ lead to the so-called geometric product \begin{align*}
e_i e_j= e_i\cdot e_j + e_i \wedge e_j, \;\;  1\leq i,j \leq m.
\end{align*}
The  definitions of inner  and geometric products then extend to blades of the grade $r$ as follows. For the inner product we put $1 \cdot e_A =0$ and
\begin{align*}
e_j \cdot e_A = e_j \cdot (e_{i_1} \wedge \cdots \wedge e_{i_r})= \sum_{k=1}^r (-1)^k B(e_j,e_{i_k}) e_{A\setminus \{i_k\}},
\end{align*}
where $e_{A\setminus \{i_k\}}$ is the blade of grade $r-1$ created by deleting $e_{i_k}$ from $e_A$. This product is also called the left contraction in literature.
For the geometric product we define
\begin{align*}
e_je_A=e_j\cdot e_A + e_j \wedge e_A.
\end{align*} 
These definitions extend linearly to the whole vector space $\Lambda(\mathbb{R}^m)$.
Thus we get an associative algebra over this vector space, the so-called real geometric algebra, denoted by $\mathbb{G}_{m}$.  Note that this algebra is naturally graded; the grade zero and grade one elements are identified with $\mathbb{R}$ and $\mathbb{R}^m$, respectively. 
Finally, we can define the norm of a blade as the magnitude of the blade $|e_A|= \sqrt{ e_A \cdot \tilde e_A }$. Note that $ e_A \cdot \tilde e_A $, where $e_A \neq 0$ is always positive in $\mathbb{G}_{m}$.

\subsection{Objects}
The vectors in $\mathbb R^m$ with the coordinates $(x_1,\dots,x_m)$ are given by
$
x= x_1 e_1 + \cdots +x_m e_m
$
and the square with respect to the geometric product $x^2=x_1^2  + \cdots  +x_m^2 \in \mathbb R$ coincides with the square of the Euclidean norm of $x$. A vector $x$ represents a one-dimensional subspace (line) $p$ in $\mathbb R^m$ given by the scalar multiples of $x$ which in $\mathbb{G}_m$ is expressed by the formula $u \in p \Leftrightarrow u \wedge x=0$. In the same way, a plane $\pi$ generated by two vectors $x$ and $y$ is represented by $x \wedge y$ in the sense  $u \in \pi \Leftrightarrow u \wedge x \wedge y=0$. In general, any $r$-dimensional subspace $V_r\subseteq \mathbb{R}^m$ is represented by a blade $A_r$ of grade $r$ such that
\begin{align} \label{NO} 
V_r=NO(A_r) = \{x \in \mathbb R^m : x \wedge A_r =0 \}.
\end{align}
 Such a representation is called the outer product null space (OPNS) representation in the literature. In particular, the whole space $\mathbb{R}^m$ is represented by a blade of maximal grade, so called pseudoscalar. Similarly one defines the inner product null space (IPNS) representation $A^*_{m-r}$ of $V_r$ as a blade of grade $m-r$  such that $x\in V_r \Leftrightarrow x \cdot A^*_{m-r} =0$. 
 The OPNS and IPNS representations are mutually dual with respect to the duality on $\mathbb{G}_m$ defined by the multiplication by pseudoscalar, namely   $$A^*=AI,$$  where $A$ is a blade and $I$ is the pseudoscalar. 
 Indeed,  one can show that $(x \wedge A)I=x \cdot (AI)$ for each vector $x \in \mathbb R^m$, in particular $$x \wedge A =0 \Leftrightarrow x \cdot A^* =0.$$

\begin{rem*}
OPNS representations of objects in $\mathbb G_3$ are summarized  in Table \ref{blades}. 
 For example,
  the plane generated by the vectors $u$, $v$ has the OPNS representation $u \wedge v$. Its IPNS representation $(u \wedge v)^*$ is a vector perpendicular to the plane.   More specifically, for the pseudoscalar $I=e_1 \wedge e_2 \wedge e_3 = e_1e_2e_3 $ we receive the usual vector product  in geometric algebra form as
 \begin{align} \label{we-cr}
 u \times v = - (u \wedge v)I.
 \end{align}
 \end{rem*}
 
 \begin{table}[h] 
 	\begin{center}
 		\begin{tabular}{ |c| c | c | c |c |} 
 			\hline
 			grade & name & blades & dimension & objects  \\
 			\hline 
 			0& scalars & 1 & 1 & numbers  \\
 			\hline
 			1 & vectors & $e_1, \: e_2, \:e_3$ & 3  & lines \\
 			\hline
 			2 & bivectors & $e_1 \wedge e_2, \: e_1 \wedge e_3, \:e_2 \wedge e_3$ & 3  & planes \\
 			\hline
 			3 & pseudoscalars & $e_1 \wedge e_2 \wedge e_3$ & 1 & volume forms \\
 			\hline
 		\end{tabular}
 	\end{center}
 	\caption{Blades of geometric algebra $\mathbb G_3$}
 \label{blades}
 \end{table}

\subsection{Transformations}
 
Let us fix a vector $n \in  \wedge^1  \mathbb R^m  \subset \mathbb G_m$ such that $n \cdot n = n^2=1$ and $x \in  \wedge^1  \mathbb R^m  \subset \mathbb G_m$ arbitrary.
The negative conjugation $-nxn$ defines the reflection with respect to the hyperplane orthogonal to $n$, 
because
\begin{align*} 
-n  x^{\perp}  n &= -(n \wedge  x^{\perp} )n=
(x^{\perp}   \wedge n )n =  x^{\perp}   n n=  x^{\perp},\\
-n x^{||}  n &=-x^{||},
\end{align*}
 where  $x=x^{||} + x^{\perp}$ is the orthogonal decomposition of $x$ with respect to $n$.
 Conjugation preserves the grades of blades and is an outermorphism $ n (u_1 \wedge \cdots \wedge u_l )n  =(n u_1 n) \wedge \cdots \wedge (n u_l n)  $ for any vectors $u_1, \dots , u_l$, thus minus the conjugation is the (anti)outermorphism depending on the dimension $m$.
 Since each rotation is a composition of two reflections,  a rotation in $\mathbb G_m$ is represented by the conjugation with respect to the geometric product of two vectors. To find a rotor between vectors $x$ and $y$ we have a nice formula at hand. 
 
 \begin{lemma*}\label{rot2} 
 Let $x$ and $y$ be the unit vectors in $\mathbb G_m$, i.e. $x,y \in \wedge^1 \mathbb G_m$, then the
 formula 
 \begin{align} \label{rot} 
 R_{xy} = \widehat{1 +  y  x },
 \end{align}  
 where the hat symbol stands for the normalization $\widehat{u} = {u}/{\sqrt{u \cdot u}},$
 defines the rotation in the plane $x \wedge y$ which maps vector $ x$ to $ y$
 and acts trivially on $(x \wedge y)^*$.
 \end{lemma*}
 \begin{proof} 
 Multiplication of two vectors $x, y \in  \wedge^1  \mathbb R^m  \subset   \mathbb G_m$, such that  $x^2=y^2=1,$ defines a multivector 
 $$
 yx = \cos(\theta) + \sin(\theta) \widehat{y\wedge x}, 
 $$
 where $\theta$ is the angle between $x$ and $y$\cite[Lemma 4.2]{per}.
The conjugation by such multivector $yx$ represents the rotation in the plane $x \wedge y$ with respect to angle $2 \theta$ in the positive way. Using standard trigonometric formulas  we can see by straightforward calculation that   
 \begin{align*}
 \begin{split} 
 R_{xy} &= \widehat{1 +   yx} = 
 \frac{1 + \cos(\theta)  + \widehat{y\wedge x} \sin(\theta)}{\sqrt{(1 +\cos(\theta))^2+\sin^2(\theta)}}
 = 
 \frac{1 + \cos(\theta)  + \widehat{y\wedge x} \sin(\theta)}{\sqrt{2 +2\cos(\theta)}} \\
 &
 = 
 \sqrt{ \frac{1 + \cos(\theta)}{2}}  +\widehat{y\wedge x} \sqrt{\frac{1-\cos^2(\theta)}{2(1 +\cos(\theta))}}
 = 
 \sqrt{ \frac{1 + \cos(\theta)}{2}}  + \widehat{y\wedge x} \sqrt{\frac{1-\cos(\theta)}{2}} \\
 &= \cos(\frac{\theta}{2}) + ( y \wedge x) \sin(\frac{\theta}{2}).
 \end{split} 
 \end{align*} 
 So $R_{xy}$ is the rotation in the plane $\widehat{x\wedge y}$ in the positive way about the angle between $x$ and $y$ so the vector $x$ goes to the vector $y$.

 Finally, $(x \wedge y)^*$ is orthogonal to $x$ and $y$ 
 and the straightforward  computation 
$ xy (x \wedge y)^*   yx = xyyx(x \wedge y)^* = (x \wedge y)^*$  proves the rest of the statement.
\end{proof}
%

%
%
%


\begin{rem}
One can see that $yxy$ is an axial symmetry with respect to $y$ and thus $x+yxy =2 \cos (\theta) y$. We can compute the square of the norm
\begin{align*}
    (1+yx)(1+xy) = 1+xy+yx+1=2+2 \cos(\theta)
\end{align*}
and with the help of the geometric product, we  compute 
\begin{align*}
    (1+yx) x(1+xy) &= (x+y) (1+xy) = (x+2y+yxy) =
    (2+ 2\cos(\theta))y 
\end{align*}
so the conjugation by \eqref{rot} maps $x$ to  $y$. 
\end{rem}

\subsection{Rotor construction}

Let $(x_1,\dots, x_{m})$ and $(y_1,\dots, y_{m})$ be a pair of bases of $\R^m$ such that 
\begin{enumerate}
\item
 $x_i \cdot x_j = y_i \cdot y_j$  for all $i,j=1,\dots , m$, i.e. all the scalar products are equal, and 
\item
 $x_1 \wedge  \cdots  \wedge  x_m
=y_1 \wedge  \cdots  \wedge  y_m$, i.e. the pseudoscalars are equal.  
\end{enumerate}
Let us remind that a complete flag $\{V\}$ in an increasing sequence of subspaces of the vector space $\mathbb R^m$ 
$$\{0\}\subset V_{1}\subset V_{2}\subset \cdots \subset V_{m}=\mathbb R^m,$$
such that $\text{dim}(V_i)=i$. 
We use the complete flags to find the explicit rotation $R$ such that $R x_i \tilde R =y_i$ for all $i=1,\dots m$. Our method can be summarized as follows:
\begin{itemize}
\item  We consider the complete flags $\{V\}$ and $\{W\}$
 by setting $V_i = \langle x_1 , \dots, x_i\rangle = NO(x_1  \wedge  \cdots \wedge x_i)$ and $W_i = \langle y_1 , \dots, y_i\rangle = NO(y_1  \wedge  \cdots \wedge y_i)$, respectively. 
 \item We map the complete flag $\{ V \}$ to the complete flag  $\{ W \}$ inductively in $m$ steps. In the $j^\text{th}$ step, we assume $V_i = W_i$ for $i>j$ and we find the rotation $R_i$ such that $R_i V_i \tilde R_i = W_i$ for $i>j-1$. 
\end{itemize}
Before we formulate the construction in detail, we need several technical lemmas.

\begin{lemma*}\label{l2}
	Let $(x_1,\dots, x_{m})$ and $(y_1,\dots, y_{m})$ be a pair of bases such that 
	\begin{enumerate}
\item 	$x_i \cdot x_j = y_i \cdot y_j$ for all $i,j=1,\dots,m$, 
	and 
	\item $x_1 \wedge  \cdots  \wedge  x_i	=y_1 \wedge  \cdots  \wedge  y_i$ for all $i=1,\dots,m$.
	\end{enumerate}
	If $\{V \}$ and $\{ W \}$ are the corresponding complete flags, respectively, and if $V_i=W_i$ for all $i=1,\dots , m,$  then $x_i=y_i$ for all $i=1,\dots ,m$.
\end{lemma*}
\begin{proof}
The equality $x_1=y_1$ holds trivially from the assumptions. Then $x_2 \cdot x_2 = y_2 \cdot y_2 $ reads that $x_2,y_2$ are of the same length,
$x_2 \cdot x_1 = y_2 \cdot y_1 $ reads that the angles between $x_1,x_2$ and $y_1,y_2$ are identical and $x_1 \wedge x_2 = y_1 \wedge y_2$ reads that they have the same orientation. Then $V_2=W_2$ and $x_1=y_1$ imply $x_2=y_2$
 and so on for all the basis vectors.  \end{proof}

\begin{lemma*}\label{l2b}
	Let $(x_1,\dots, x_{i},z)$ and $(y_1,\dots, y_{i},z)$ be a pair of sets of independent vectors such that 
	$x_1 \wedge  \cdots  \wedge  x_i \wedge	z =y_1 \wedge  \cdots  \wedge  y_i \wedge	z$. 
	If $NO(x_1 \wedge  \cdots  \wedge  x_i) = NO(x_1 \wedge  \cdots  \wedge  x_i)$  then $x_1 \wedge  \cdots  \wedge  x_i	=y_1 \wedge  \cdots  \wedge  y_i$.
\end{lemma*}
\begin{proof}
The independence of the sets of vectors imply $x_1 \wedge  \cdots  \wedge  x_i \wedge z \neq 0$, $y_1 \wedge  \cdots  \wedge  y_i \wedge z \neq 0$.
If $NO(x_1 \wedge  \cdots  \wedge  x_i) = NO(y_1 \wedge  \cdots  \wedge  y_i)$ then 
$x_1 \wedge  \cdots  \wedge  x_i = \beta  x_1 \wedge  \cdots  \wedge  x_i$ for $\beta \in \mathbb R$.

Then
\begin{align*}
    x_1 \wedge  \cdots  \wedge  x_i \wedge	z &=y_1 \wedge  \cdots  \wedge  y_i \wedge	z  \\
    (x_1 \wedge  \cdots  \wedge  x_i -y_1 \wedge  \cdots  \wedge  y_i) \wedge	z &=0
    \\
    (1-\beta ) x_1 \wedge  \cdots  \wedge  x_i \wedge	z &=0
\end{align*}
Finally $\beta=1$ and 
$x_1 \wedge  \cdots  \wedge  x_i = y_1 \wedge  \cdots  \wedge  y_i$.
\end{proof}


\begin{lemma*} \label{l3}
Consider two complete flags $\{V\}$ and $\{W\}$ in $\R^m$ and $i\leq m$ such that
  $V_{j} = W_{j}$ for $j>i$.
 	The rotor $R_i$ between the hyperplanes 	$V_i \oplus V_{i+1}^{\perp}$ and $W_i \oplus W_{i+1}^{\perp}$  constructed by the formula \eqref{rot} maps $V_i$  to $W_i$.  
\end{lemma*}
\begin{proof}
The property $V_i \subset V_{i+1}$ implies that $V_{i+1}^{\perp} \subset V_{i}^{\perp}$ and thus $V_i \oplus V_{i+1}^{\perp}$ is a hyperplane equipped with an  orthogonal decomposition. Recall that $V_{i+1}=W_{i+1}$, and thus $V^\perp_{i+1}=W^\perp_{i+1}$.
Any rotation preserves an orthogonal decomposition and thus $R_i$
acts as the identity 
on $V^\perp_{i+1}=W^\perp_{i+1}$
, so it maps $V_i$ to $W_i$. 
\end{proof}
We use all these Lemmas to
provide a constructive proof of the following theorem.

\begin{thm}\label{thm}
Let $(x_1,\dots, x_{m})$ and $(y_1,\dots, y_{m})$ be a pair of bases  of $\R^m$ such that 
\begin{enumerate}
\item $x_i \cdot x_j = y_i \cdot y_j$ for all $i,j=1 ,\dots, m $, and 
\item $x_1 \wedge  \cdots  \wedge  x_m =y_1 \wedge  \cdots  \wedge  y_m$.
\end{enumerate}
Then we can construct a rotor $R$ such that  $Rx_i  \tilde R = y_i$ for all $j,i= 1,\dots m$.
\end{thm}
\begin{proof}
 Let $\{ V \}$ and $\{ W \}$ be a pair of the corresponding complete flags
$V_i=NO(x_1 \wedge \cdots \wedge x_i)$ and $W_i=NO( y_1 \wedge \cdots \wedge y_i)$. We construct a rotor  $R=R_1 \cdots R_m$
mapping the complete flag $\{ V \}$ to the complete flag $\{ W \}$ so that $V_i=W_i$ for all $i=1,\dots m$. The result on bases $x_i, y_i$ then follows by Lemmas \ref{l2} and \ref{l2b}.

We define $R_m$ as the identity and proceed inductively. 
It follows from Lemma \ref{rot2} that there is a rotation $R_{m-1}$  between the hyperplanes $V_{m-1} \oplus V_m^{\perp} \cong V_{m-1}$ and $W_{m-1} \oplus W_m^{\perp} \cong W_{m-1}$ which maps the complete flag $\{ V \}$ to the complete flag $\{ R V  \tilde R \}$ in such a way that $W_{m-1}=R V_{m-1} \tilde R$, where $R= R_{m-1} R_m$.

As the induction step, we consider the rotor $R=R_{j} \cdots R_m$ such that 
$RV_i  \tilde R = W_i$ for all indices $i \geq j$. According to Lemma \ref{l3}, the rotation $R_{j-1}$ between the hyperplanes $(R V_{j-1}  \tilde R )\oplus (R V_j  \tilde R )^{\perp}$ and $W_{j-1} \oplus W_j^{\perp} $  maps the complete flag $\{ R V   \tilde R\}$ to the complete flag $\{ R_j R V \tilde R \tilde{R}_j \}$ in such a way that $W_{i}= R_j R V_{i} \tilde R \tilde{R}_j$ for all $ i \geq j-1 $.
 
After $m$ steps, the rotor $R=R_1 \cdots R_m$
maps the complete flag $\{ V \}$  to the complete flag $\{ W \}$ in such a way that $V_i=W_i$ for all $i=1,\dots, m$
and so $Rx_i  \tilde R = y_i$ for all $j,i= 1,\dots, m$ because of Lemma \ref{l2}.
\end{proof} 

The explicit construction in the proof of Theorem \ref{thm} gives us  the following algorithm.

\begin{algorithm}{Calculate the rotor $R= R_1\dots R_m$}
\begin{algorithmic} \label{algor}
\Require  $x_i \cdot x_j = y_i \cdot y_j$ and $x_1 \wedge  \cdots  \wedge  x_m=y_1 \wedge  \cdots  \wedge  y_m$
\Ensure $y_i = R x_i  \tilde R$
\For{$m > i > 0$}
\State $V_i \leftarrow x_1 \wedge \cdots \wedge x_i$
\State $W_i \leftarrow y_1 \wedge \cdots \wedge y_i$
\EndFor
\State $R \leftarrow Id$
\For{$m > i > 0$}
\State $V_i \leftarrow R V_i  \tilde R$
\State $H_V \leftarrow V_i \wedge W_{i+1}^*$
\State $H_W \leftarrow W_i \wedge W_{i+1}^*$
\State $R_i \leftarrow \widehat{1+\hat{H}_V^* \hat{H}_W^*}$
\State $R \leftarrow R_iR$
\EndFor
\end{algorithmic}
\end{algorithm}

\section{Nilpotent control problems in GA approach} \label{GAapproach}

We use the symmetries of $SO(3)$ to define an equivalence relation on the set of geodesics passing through the origin, see Proposition \ref{akce}. We find a convenient representative of any equivalence class and describe the moduli space in the language of GA.

\subsection{Geodesics of $(3,6)$}

Since the vector product $x \times h$ coincides with minus the dual of wedge product
$x \wedge h$ according to \eqref{we-cr},  the horizontal system  \eqref{rr1} can be written in the form 
\begin{align}
\begin{split} \label{rr1GA}
\dot x &= h, \\
\dot z &= -{1 \over 2}x \wedge h 
\end{split}
\end{align} 
 where $x \in \wedge^1 \mathbb R^3$ represents a line and $z \in \wedge^2 \mathbb R^3$ represents a plane in $\mathbb R^3$. In this way we see the geodesics as curves in the geometric algebra $\mathbb G_3$.

\begin{prop*} \label{geodesics}
Each arc--length parameterized sub--Riemannian geodesic satisfying the initial condition $x_i(0)=0,z_i(0)=0$, $i=1,2,3$ is equivalent to a curve in $M\cong \wedge^1 \mathbb R^3 \oplus \wedge^2 \mathbb R^3 \subset \mathbb G_3 $ and, up to the action of a suitable  $R \in SO(3),$ it takes form
\begin{align} \begin{split} \label{Gx}
 q(t)&=x(t)+z(t)=\frac{D}{K}(1-\cos(Kt))e_1+\frac{D}{K}\sin(Kt)e_2+C_3te_3 \\
& - \frac{D^2}{2K^2}(Kt-\sin(Kt)) e_1 \wedge e_2  \\ &
 -\frac{C_3D}{2K^2} (Kt-2\sin(Kt)+Kt\cos(Kt)) e_3 \wedge e_1 \\ &+\frac{C_3D}{2K^2}(2-Kt\sin(Kt)-2\cos(Kt)) e_2 \wedge e_3,
\end{split}
 \end{align}
where $K>0$ and $D,C_3$  satisfy the level set equation
$
D^2+C_3^2= 1.
$

\end{prop*}
\begin{proof}
The solution to the vertical system \eqref{h36} can be rewritten as 	
\begin{align}
\begin{split}  
\label{hh1}
h (t)
= D\sin(Kt) \bar v_1
+ D\cos(Kt) \bar v_2
+ C_3  v_3
\end{split},
\end{align}
where we denote $D = \sqrt{C_1^2+C_2^2},$ and
the orthonormal vectors $\bar{v}_1,\bar{v}_2$ 
are obtained by the rotation of orthonormal vectors $v_1,v_2$
as
 $$\bar v_1=\frac{1}{\sqrt{C_1^2+C_2^2}}(-C_1v_1+C_2v_2),\;
 \bar v_2=\frac{1}{\sqrt{C_1^2+C_2^2}}(C_2v_1+C_1v_2).$$ 
Thus the vectors 
$\bar v_1, \bar v_2, v_3$ are orthonormal  with respect to the Euclidean metric on $\mathbb R^3$.
So, there is an orthogonal matrix $R \in SO(3)$ that aligns vectors $\bar v_1, \bar v_2,  v_3$ with the standard basis of $\R^3$. Thus we get
$$ \bar v_1 = R
e_1
,\ \ \  \bar v_2 = Re_2,  \ \ v_3 = Re_3,
$$
where $e_1,e_2$ and $e_3$ are the elements of the standard Euclidean basis of $\mathbb R^3$. 
According to \eqref{action}, the rotor $R$ defines a representative of a geodesic class 
$(R^Tx(t),R^Tz(t))$ which is a solution to \eqref{rr1GA}
for $h (t)
= D\sin(Kt) e_1
+ D\cos(Kt) e_2
+ C_3  e_3$.
The solution \eqref{Gx} then follows by a direct integration when the initial condition is applied. Equation for the level set follows from the definition of $D$.
\end{proof}

The action of $SO(3)$ on $M \cong \mathbb R^6$ given by equation  \eqref{action} 
defines a moduli space $M/SO(3)$. We see $M$ as a subset of $\mathbb G_3$ and the group $SO(3)$ is represented by rotors instead of matrices, which act on $M$ by conjugation.
The action preserves the vector and bivector parts, inner product, norm and dualization with respect to $*$.
We can see the elements of $M$ as the pairs consisting of  lines and planes. The natural invariants are the norms of lines' directional vectors, norms of the planes' normal vectors and angles between these pairs of vectors. Square norm of the normal vector of the plane $z^* \cdot z^*$ is $-z \cdot z$.  Scalar product between the directional vector of the line $x$ and the normal vector of the plane $z$  can be rewritten as $(x \wedge z )^*$ because $(x \cdot z^*)^* =  x \wedge z  $ and $x \cdot z^* = (x \wedge z )^*$.  Altogether, we consider three invariants  
\begin{itemize}
\item the square norm of the vector $x$, i.e. $x \cdot x$,
\item the square  norm of the bivector $z$, i.e. $z \cdot z$,
\item the element $(x \wedge z )^*$ ,
\end{itemize} 
where $\cdot$ coincides with the inner product on $\mathbb G_3$. 
In particular, these invarians form a coordinate system on the moduli space $M/SO(3)$.

\begin{prop*} \label{prop4}\label{invp}
Each geodesic starting at the origin defines a curve in the moduli space $M/SO(3)$, 
which is determined by the invariants in the following way 
\begin{align}
\begin{split}
x \cdot x=& -\frac{2D^2}{K^2}(\cos(Kt) - 1) + C_3^2 t^2\\
z \cdot z =&  - \frac{D^{2}}{4 K^4} 
((4C_3^2K^2 - 4C_3^2 - D^2)\cos(Kt)^2 \\ &+ 2KC_3^2(2t(K - 1)\sin(Kt) + t^2K - 4)\cos(Kt) 
\\ &- 2Kt(4C_3^2 + D^2)\sin(Kt) + t^2(2C_3^2 + D^2)K^2 + D^2 + 8C3^2)
\\
(x \wedge z )^*=& \: \frac{D^{2}{C_3}}{2 K^3}( 
(-2K + 2) \cos(Kt)^2 + (2K + 2)\cos(Kt) + K^2t^2 + Kt\sin(Kt) - 4)
\end{split}
\label{inv} 
\end{align}
\end{prop*}
\begin{proof}
Follows directly from \eqref{Gx}.
\end{proof}

\subsection{Geodesics of $(4,7)$}
The base system \eqref{xl} can be seen as a system in geometric algebra $\mathbb G_4$ 
\begin{align}
\begin{split} \label{xlGA}   
	 \dot x + \dot \ell &= h_0+\bar{h},  \\
	 \dot y  &
	  =-x  \wedge \bar{h} - \ell \wedge  h_0,
	  \end{split}
\end{align}
where we assume that $x$ and $h_0$ are collinear with $e_1$ and $\ell,\bar{h}$ in the subspace generated by $e_2,e_3,e_4.$ The form of the second equation implies that $y$ is given by minus the wedge product of $e_1$ and a vector from this subspace. Hence the solution $y(t)$ can be viewed as a curve of planes in $\mathbb{G}_4.$


\begin{prop*} \label{geo}
Each arc--length parameterized sub--Riemannian geodesic 
satisfying the initial condition $x(0)=0,\ell_i(0)=0,y_i(0)=0,i=1,2,3$
is equivalent to a curve in $N\cong \wedge^1 \mathbb R^4 \oplus \wedge^2 \mathbb R^4 \subset \mathbb G_4 $ and,  up to the action of suitable  $R \in SO(3),$ it
takes form
 \begin{align}
 \begin{split} 
\label{newgeo}
 q(t)&=x(t)+\ell(t)+y(t)=(C_1\cos(Kt)+C_2\sin(Kt)-C_1) e_1 \\
 & +(C_1\sin(Kt)-C_2\cos(Kt)+C_2)
 e_2 + C t e_3
 \\ &+\frac{1}{2}(C_1^2+C_2^2)(tK-\sin(Kt)) e_1 \wedge e_2 \\ &+
 \frac{C}{2K}((2C_1-C_2Kt)\sin(Kt)-(C_1Kt+2C_2 )\cos(Kt) +2C_2-tC_1K)  e_1 \wedge e_3
 ,
 \end{split}
 \end{align}
where $K>0$ and the constants $C_1, C_2,C$  satisfy the level condition
$
K^2(C_1^2+ C_2^2) +C^2 =1.
$
\end{prop*}
\begin{proof} 
According to the vertical system \eqref{h1}, the vector $\bar{h}(t)$ lies in the subspace generated by the vectors $r_1,r_2$ for any $t$. Since the vectors $r_1$ and $r_2$ are orthonormal, there is an orthogonal matrix $R \in SO(3)$ that aligns these vectors  with the second and third vector of the standard basis of $\R^3$, i.e 
$$ r_1 = R e_2,\ \ \ r_2 = Re_3. $$ Due to the symmetry of this system, see \eqref{action2}, this rotor defines a representative of the geodesic class $(x(t),R^T \ell (t),R^T y (t) )$ which is the solution to the horizontal system \eqref{xl} for
\begin{align*}
h_0&=K(C_2\cos(Kt)-C_1\sin(Kt)) \\ 
\bar h(t) &=K(C_2\sin(Kt)+C_1\cos(Kt))
e_1
+Ce_2
\end{align*}
or, equivalently, a curve in $\mathbb{R}^4\oplus\Lambda^2\mathbb{R}^4\in\mathbb{G}_4$ given by the solution of \eqref{xlGA}. By direct integration of this equation and by imposing the initial conditions, we get  the formula \eqref{newgeo} for the solution.
\end{proof}

%

The action of $SO(3)$ on $N \cong \mathbb R^7$ given by  \eqref{action2}, 
defines a moduli space $N/SO(3)$. We see $N$ as a subset of $\mathbb G_4$ and the group $SO(3)$ is represented by rotors instead of matrices, which act on $N$ by the conjugation.
The action preserves the vector and bivector part, the split $x+ \ell$, inner product, norm and dualization with respect to $*$.
The orbits of this action are determined by natural invariants. For the same reason as in the case of $(3,6)$ and due to the invariant split, we have three invariants as follows
\begin{itemize}
    \item the value of the coordinate $x$,
	\item the square of the norm of the vector $\ell$, i.e. $\ell\cdot\ell$,
	\item the square of the norm of the bivector $y$, i.e. $y\cdot y$.
\end{itemize}
We need one more invariant for the dimensional reasons but the element $(\ell \wedge y )^*$ is not scalar but vector. On the other hand,  $(\ell \cdot y )$ is a multiple of the vector $e_1$, so the value of $(\ell \cdot y)e_1$ is a scalar. As the last invariant we consider 
\begin{itemize}
    \item the value of $(\ell \cdot y)e_1$.
\end{itemize}
These form the coordinate system on the moduli space $N/SO(3)$.

\begin{prop*} \label{prop5}
Each geodesic starting at the origin defines a curve in the moduli space $N/SO(3),$ 
which is determined by the invariants in the following way

\begin{align}
\begin{split}
x &= C_1 (\cos Kt - 1) + C_2 \sin Kt, \\
\ell \cdot \ell &= (C_1 \sin Kt + C_2 (1-\cos Kt ))^2+ (C t)^2,  \\
(\ell \cdot y)e_1 &=\frac12\Big( (C_1^2+C_2^2)
\big(C_1 \sin Kt + C_2 (1-\cos Kt ) \big)(Kt - \cos Kt) \\
&+\frac{C^2}{K}t  \big(C_1 (2 \sin Kt  - Kt \cos Kt - Kt   ) \\ &+ 
C_2 (2-2 \cos Kt  - Kt \sin Kt  ) \big)\Big),\\
y \cdot y &=\frac14 \Big((C_1^2+C_2^2)^2(Kt - \cos Kt)^2+ \frac{C^2}{K^2} \big( C_1 (2 \sin Kt  - Kt \cos Kt - Kt   ) \\ &+ C_2 (2-2 \cos Kt  - Kt \sin Kt  ) \big)^2\Big).
\label{inv2}
\end{split}
\end{align}
\end{prop*}
\begin{proof}
Follows directly from \eqref{newgeo}.
\end{proof}

\section{Examples}

In the sequel, we present two examples of controls based on the symmetries in geometric algebra approach. We have the following scheme based on Algorithm
\ref{algor}.  
\begin{enumerate}
\item For the target point $q_t$ compute the invariants of the chosen particular control  system \eqref{control}.
\item Solve the system of non--linear equations  \eqref{inv} or \eqref{inv2} in the moduli space.
\item Find the family of curves \eqref{Gx} or \eqref{newgeo} going from the origin to the same point $q_o$ that belongs to the same $SO(3)$ orbit of $q_t$.
\item Find $R \in SO(3)$, such that $R(q_o) =q_t$. 
\item Apply $R$ on the set of curves \eqref{Gx} or \eqref{newgeo} to get a family of curves going from the origin to the target point  $q_t$.
\end{enumerate}
The explicit calculations were acquired using a CAS system Maple similarly  to the paper \cite{hzn}.

\subsection{Example in $(3,6)$}
Our goal is to find the geodesic going from the origin to the  target point 
$$ q_t= (x_t,z_t)=  2 e_1-e_2+3 e_3+e_1 \wedge e_2-2 e_1 \wedge e_3-2 e_2 \wedge e_3
$$
using the invariants \eqref{inv} in the target point. We have $$x \cdot x = 14,  \: \: z \cdot z =-9 , \:\: (x \wedge z)^*  =3$$ and together with the level set condition we get the system with the invariants at $q_t$. We solve the system numerically in Maple and present the solution with rounding up to four decimal digits  
\begin{align} 
C_3 &= 0.7252, D =0.6885,  K =0.9886 \label{koef} \\
 t &= 5.0236
 \end{align} 
 Using the constants \eqref{koef}, we get the geodesic in the moduli space from the origin to the point $q_o$ in the form
\begin{align*}  
q&= (x,z)=
(0.6965(1-\cos(0.9886 t))) e_1+0.6965 \sin(0.9886 t) e_2+0.7252 t e_3  \\ 
&-(0.2425 (0.9886 t-\sin(0.9886 t)))e_1 \wedge e_2 \\ &+(0.2555 (0.9886 t-2 \sin(0.9886 t)+0.9886 t \cos(0.9886 t))) e_1 \wedge e_3 \\
&+(0.2555 (2-0.9886 t \sin(0.9886 t)-2 \cos(0.9886 t))) e_2 \wedge e_3
 \end{align*}
and at the time $t = 5.0236$ we reach the point
\begin{align}
\begin{split} 
q_o& =0.5216 e_1-0.6741 e_2+ 3.643 e_3 -1.439 e_1 \wedge e_2+2.082 e_1 \wedge e_3+1.611 e_2 \wedge e_3
\end{split}
\end{align}
We are looking for the rotor which maps the multivector $q_o$ on the multivector $q_t$.
We consider the complete flags 
\begin{align*} &\{0\} \subset NO( x_t)  \subset NO(  x_t \wedge z_t^*)
\subset NO(  z_t \wedge z_t^*) \cong \mathbb R^3,\\
 &\{0\} \subset NO( x_o)  \subset NO(  x_o \wedge z_o^*)
\subset NO(  z_o \wedge z_o^*) \cong \mathbb R^3.
\end{align*}
We set $R_m =R_3= \text{id}$ and map the plane $x_o \wedge z_o^*$ to the plane 
$x_t \wedge z_t^*$ by the rotor $R_{m-1}=R_2$ according to the formula \eqref{rot}. Explicitly, 
$$ R_2 := 0.1334 +0.7083 e_1 \wedge e_2- 0.5483 e_1 \wedge e_3-0.4242 e_2 \wedge e_3$$
and we can map the multivector $q_o$ on the multivector $q_s= (x_s,z_s)=R_2 q_o \tilde {R_2}$ in such a way that $x_s$ and $z_s$ lie in the plane  $x_o \wedge z_o^*$.
Explicitly,
$$ q_s= (x_s,z_s)=
-2.8510 e_1+2.3208 e_2 -0.6956 e_3
-2.641 e_1 \wedge e_2+1.2523 e_1 \wedge e_ 3+0.6767 e_2 \wedge e_3. $$
Finally, we  map the plane $x_s \wedge (x_s \wedge z_s)^*$ to the plane 
$x_t \wedge (x_t \wedge z_t)^*$ by rotor
$$R_{m-2}= R_1=
0.3727+0.1716 e_1\wedge e_2+0.6863 e_1\wedge e_3-0.6005 e_2\wedge e_3
.$$
Altogether, we found the rotor $R=R_1R_2R_3$ and, when applied on \eqref{Gx},
we got a geodesic going from the origin to the point  $q_t$ in the form 

\begin{align*}
q&=(x,z) = 
(-0.5302+0.4673 t+0.5302 \cos(0.9886 t)-0.05136 \sin(0.9886 t)) e_1 \\
&-(0.008(22.124+36.347 t-22.124 \cos(0.9886 t)+76.628 \sin(0.9886 t))) e_2 \\
&+(-0.4156 \cos(0.9886 t)+0.4723 t+0.4156-0.3267 \sin(0.9886 t)) e_3 \\ &-(0.2 (-1.5244+0.7536 t \sin(0.9886 t)+0.4086 \sin(0.9886 t)  
+1.5244 \cos(0.9886 t) \\ &-0.5923 t \cos(0.9886 t)+0.1884 t)){e_1 \wedge e_2} \\
&+(-0.3184 t+0.1299-0.2223 t \cos(0.9886 t)-0.1299 \cos(0.9886 t) \\
&-0.06419 t \sin(0.9886 t)+0.547 \sin(0.9886 t)){e_1 \wedge e_3}\\
&+(-0.389+0.1923 t \sin(0.9886t)-0.1358 t +0.0186t \cos(0.9886 t) \\ &+0.389 \cos(0.9886 t)+0.1186 \sin(0.9886 t))  e_2 \wedge e_3.
\end{align*} 

In Figure \ref{ff1} we present the trajectories $(x_1,x_2,x_3)$ 
and $(z_1,z_2,z_3)$, respectively. 

 \begin{figure}[h]
 \begin{center}
 \includegraphics[height=50mm]{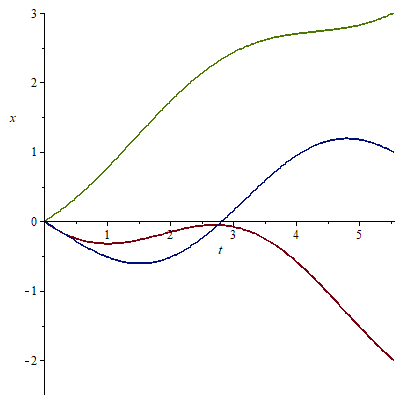} \ \ \ \  \ \ \
 \includegraphics[height=50mm]{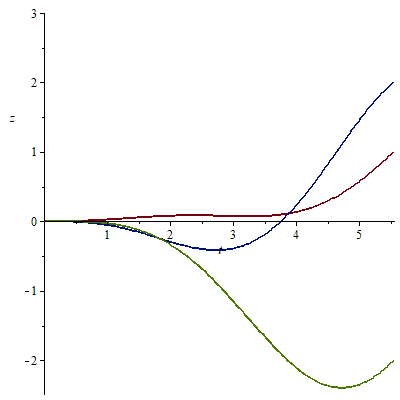}
 \caption{Trajectories $(x_1,x_2,x_3)$ and $(z_1,z_2,z_3)$}
 \label{ff1}
 \end{center}
 \end{figure}

\subsection{Example in $(4,7)$}
Our goal is to find the geodesic going from the origin to the  target point 
$$ q_t= (x_t,\ell_t, y_t)=  e_1 +2 e_2 +e_3 +3 e_4 -  e_1 \wedge e_2 +
2  e_1 \wedge e_3 + 2 e_1 \wedge e_4
$$
using the invariants \eqref{inv2} at the target point. We compute 
$$x  = 1,  \: \: \ell \cdot \ell =14 , \:\: y \cdot y  =-9, \:\: (\ell \cdot y) e_1 =-6,$$ and together with the level set condition, we get the system
with the invariants at $q_t$. We solve the system numerically in Maple and we present the solution with the constants rounded up to four decimal digits as follows 
\begin{align} 
C &= 0.6126, C_1 = -0.7816, C_2 = -0.5324, K = 0.8358 \label{koef2}, \\
 t &= 6.0748.
 \end{align} 
 Using the constants \eqref{koef2} we get a geodesic in the moduli space from the origin to the point $q_o$ in the form
\begin{align*}
	\begin{split} \label{Gx2}   
q= (x,\ell,y)&= 
   (-0.7816 \cos(0.8358 t) - 0.5324 \sin(0.8358 t) + 0.7816) e_1\\ &+
   (-0.7816 \sin(0.8358 t) + 0.5324 \cos(0.8358 t) - 0.5324) e_2\\
      &+ 0.6126 t e_3
+ 0.4471 (0.8358 t - \sin(0.8358 t)) e_1  \wedge e_2 \\ &+ 0.3665 \big((0.4449 t
    - 1.563) \sin(0.8358 t) + (0.6533 t + 1.065) \cos(0.8358 t) \\
    &- 1.065 + 0.6533 t\big) e_1 \wedge e_3
\end{split}
\end{align*}
and at the time $t = 6.0748$, we reach the point
\begin{align}
\begin{split} 
q_o& = (x_o,\ell_o,y_o)= e_1+ 0.3878 e_2 + 3.722 e_3 + 2.688 e_1 \wedge e_2 + 1.332 e_1 \wedge e_3.
\end{split}
\end{align}
We are looking for the rotor which maps the multivector $q_o$ on the multivector $q_t$.
We shall consider the complete flags starting with the line $NO(\ell)$ and ending with the space $NO(\ell \wedge y)$. To find the middle one, we can use the projection of the line   $NO(\ell)$ onto the plane  
$NO(\ell \wedge y)$. Thus we get \begin{align*} &\{0\} \subset NO( \ell_t)  \subset NO(  \ell_t \wedge (\ell_t \wedge y_t^*)^*) = NO(  \ell_t \wedge (\ell_t \cdot y_t)) \\  &
\subset NO(  \ell_t \wedge  y_t) \subset NO(  y_t \wedge  y_t^*)  \cong \mathbb R^4, \\
 &\{0\} \subset NO( \ell_o)  \subset NO(  \ell_o \wedge (\ell_o \wedge y_o^*)^*) = NO(  \ell_o \wedge (\ell_o \cdot y_o)) \\ 
& \subset NO(  \ell_o \wedge  y_o) \subset NO(  y_o \wedge  y_o^*)  \cong \mathbb R^4.
\end{align*}
First, we map the hyperplane $\ell_o \wedge y_o$ to the hyperplane 
$\ell_t \wedge y_t$ by the rotor
$R_3$ according to the formula \eqref{rot}.
We obtain
\begin{align*} 
R_3 &:= 0.4863 - 0.4335 e_2 \wedge e_4 - 0.7587 e_3 \wedge e_4.
\end{align*}
The next step is to map the hyperplane 
$NO(  \ell_o \wedge (\ell_o \cdot y_o) \wedge  (\ell_o \wedge y_o)^* ) $ to the hyperplane 
$NO(  \ell_t \wedge (\ell_t \cdot y_t) \wedge  (\ell_t \wedge y_t)^* ) $
by the rotor
$R_2$ according to the formula \eqref{rot}. Explicitly, 
\begin{align} 
R_2 &:= 
 0.0387 + 0.9993 e_2 \wedge e_3.
\end{align}
Finally, we map the hyperplane  $NO(  \ell_o  \wedge  (\ell_o  \wedge  (\ell_o \wedge y_o^*)^*)^* ) $ to the hyperplane $NO(  \ell_t  \wedge  (\ell_t  \wedge  (\ell_t \wedge y_t^*)^*)^* )$  
by the rotor
$R_1$ according to the formula \eqref{rot}. It turns out that $R_1= 1.$
Altogether, $R=R_1R_2R_3$ and, when applied on \eqref{Gx2}, we get a geodesic going from the origin to the point  $q_t$ as
\begin{align*}
q=(x,\ell,y) &=(- 0.7816\cos (  0.8358t ) - 0.5324\sin (  0.8358
t ) + 0.7816
)e_1
\\ &+
(0.5261\sin (  0.8358t ) - 0.3584\cos (  0.8358
t ) + 0.3946t+ 0.3584
)e_2\\&+
(- 0.4749\sin (  0.8358t ) + 0.3235\cos (  0.8358
t ) + 0.1235t- 0.3235
)e_3
\\&+
((  0.2513+ 0.1542t ) \cos (  0.8358t ) +
 ( - 0.06803+ 0.1050t ) \sin (  0.8358t ) \\& -
 0.2514- 0.09731t
 )e_1 \wedge e_2
\\ &+
((  0.07868+ 0.04827t ) \cos (  0.8358t ) +
 ( - 0.3871+ 0.03287t ) \sin (  0.8358t ) \\&-
 0.07869+ 0.2753t
 )e_1 \wedge e_3\\
 &+ ((  0.2880+ 0.1767t ) \cos (  0.8358t ) +
 (  0.1203t- 0.6112 ) \sin (  0.8358t )  \\ &+
 0.3342t- 0.2880
)e_1 \wedge e_4.
\end{align*} 
In Figure \ref{ff} we present trajectories $x$, $(\ell_2,\ell_3, \ell_4)$  and $(y_1,y_2,y_3)$, respectively. 
 \begin{figure}[h]
 \begin{center}
 \includegraphics[height=40mm]{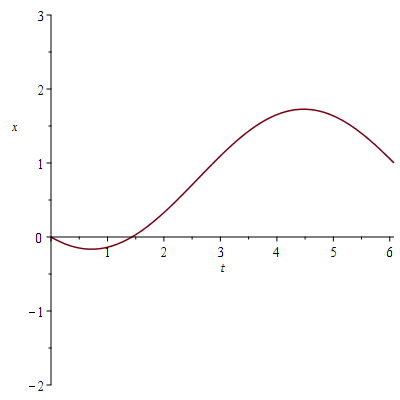} \ \
 \includegraphics[height=40mm]{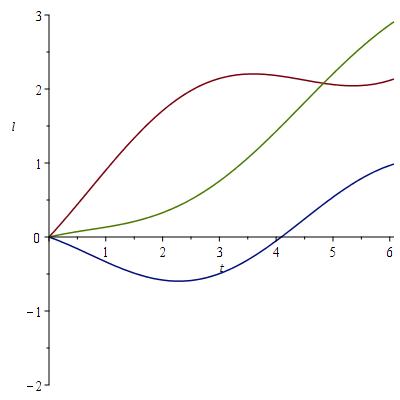}\ \
 \includegraphics[height=40mm]{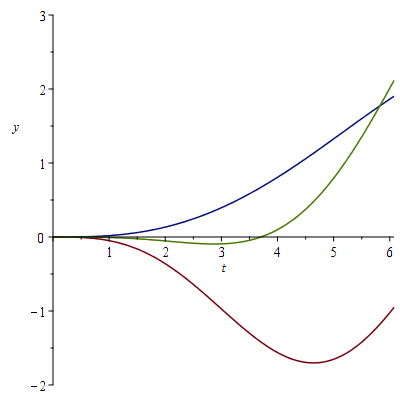}
 \caption{Trajectories $x$, $(\ell_2,\ell_3,\ell_4)$ and $(y_1,y_2,y_3)$}
 \label{ff}
 \end{center}
 \end{figure}

\section{Conclusion}
We presented the use of geometric algebra for the control systems invariant with respect to the orthogonal transformations. The main contribution of GA lies in a construction of the rotor between two bases of a vector space based only on algebraic computations in  a chosen GA. This allows us to use the geometric objects effectively and, analogously to quaternions, the implementations are faster than the usual computations with matrices.
We assessed an algorithm and illustrated its use on two particular examples with filtration $(3,6)$ corresponding to a trident snake robot control,  and $(4,7)$ corresponding to the control of a trident snake with flexible leg. All calculations were acquired using Maple packages Clifford \cite{Cliff} and DifferentialGeometry \cite{dg}.

\subsection{Acknowledgements} 
This work does not have any conflicts of interest.
The first three authors were supported by the grant no. FSI-S-20-6187. Fourth author is supported by the grant no. 20-11473S Symmetry and invariance in analysis, geometric modeling and control theory from the Czech Science Foundation. 
Finally, we thank the referee for valuable comments.

\end{document}